\documentclass[11pt]{amsart}

\usepackage[T1]{fontenc}
\usepackage{lmodern}
\usepackage{microtype}
\usepackage{amsmath,amssymb,amsthm,mathtools}
\usepackage{tikz-cd}
\usepackage{tikz}
\usepackage{booktabs,tabularx,array}
\usepackage{enumitem}
\usepackage{aliascnt}
\usepackage{xcolor}
\definecolor{refblue}{RGB}{0,0,150}
\definecolor{citered}{RGB}{180,0,0}
\usepackage[colorlinks=true,linkcolor=refblue,citecolor=citered,urlcolor=citered]{hyperref}

\setlist[enumerate]{label=(\arabic*),leftmargin=2.1em}
\setlist[itemize]{leftmargin=1.8em}
\allowdisplaybreaks
\numberwithin{equation}{section}

\newtheorem{theorem}{Theorem}[section]
\newaliascnt{proposition}{theorem}
\newtheorem{proposition}[proposition]{Proposition}
\aliascntresetthe{proposition}
\newaliascnt{lemma}{theorem}
\newtheorem{lemma}[lemma]{Lemma}
\aliascntresetthe{lemma}
\newaliascnt{corollary}{theorem}
\newtheorem{corollary}[corollary]{Corollary}
\aliascntresetthe{corollary}
\newaliascnt{assumption}{theorem}
\newtheorem{assumption}[assumption]{Assumption}
\aliascntresetthe{assumption}
\newaliascnt{question}{theorem}

\aliascntresetthe{question}
\theoremstyle{definition}
\newaliascnt{definition}{theorem}
\newtheorem{definition}[definition]{Definition}
\aliascntresetthe{definition}
\newaliascnt{example}{theorem}
\newtheorem{example}[example]{Example}
\aliascntresetthe{example}
\theoremstyle{remark}
\newaliascnt{remark}{theorem}
\newtheorem{remark}[remark]{Remark}
\aliascntresetthe{remark}

\usepackage[nameinlink,noabbrev]{cleveref}
\crefname{theorem}{Theorem}{Theorems}
\Crefname{theorem}{Theorem}{Theorems}
\crefname{proposition}{Proposition}{Propositions}
\Crefname{proposition}{Proposition}{Propositions}
\crefname{lemma}{Lemma}{Lemmas}
\Crefname{lemma}{Lemma}{Lemmas}
\crefname{corollary}{Corollary}{Corollaries}
\Crefname{corollary}{Corollary}{Corollaries}
\crefname{assumption}{Assumption}{Assumptions}
\Crefname{assumption}{Assumption}{Assumptions}
\crefname{question}{Question}{Questions}
\Crefname{question}{Question}{Questions}
\crefname{definition}{Definition}{Definitions}
\Crefname{definition}{Definition}{Definitions}
\crefname{example}{Example}{Examples}
\Crefname{example}{Example}{Examples}
\crefname{remark}{Remark}{Remarks}
\Crefname{remark}{Remark}{Remarks}

\newcommand{\E}{\mathcal E}
\newcommand{\Acat}{\mathcal A}
\newcommand{\Xcat}{\mathcal X}
\newcommand{\Ccat}{\mathcal C}
\newcommand{\Scat}{\mathcal S}
\newcommand{\Pcat}{\mathcal P}
\newcommand{\Fac}{\operatorname{Fac}}
\newcommand{\Sub}{\operatorname{Sub}}
\newcommand{\Filt}{\operatorname{Filt}}
\newcommand{\add}{\operatorname{add}}
\newcommand{\Proj}{\operatorname{Proj}}
\newcommand{\rad}{\operatorname{rad}}
\newcommand{\topm}{\operatorname{top}}
\newcommand{\ann}{\operatorname{ann}}
\newcommand{\End}{\operatorname{End}}
\newcommand{\Hom}{\operatorname{Hom}}
\newcommand{\Ext}{\operatorname{Ext}}
\newcommand{\Tor}{\operatorname{Tor}}
\newcommand{\pd}{\operatorname{pd}}
\newcommand{\gldim}{\operatorname{gldim}}
\newcommand{\findim}{\operatorname{findim}}
\newcommand{\Findim}{\operatorname{Findim}}
\newcommand{\dell}{\operatorname{dell}}
\newcommand{\ddell}{\operatorname{ddell}}

\newcommand{\gddell}{\operatorname{gddell}}
\newcommand{\depth}{\operatorname{depth}}
\newcommand{\grade}{\operatorname{grade}}
\newcommand{\Fddell}{\operatorname{Fddell}}
\newcommand{\Mod}{\operatorname{Mod}}
\newcommand{\modu}{\operatorname{mod}}

\newcommand{\Ker}{\operatorname{Ker}}
\newcommand{\op}{\mathrm{op}}

\newcommand{\stable}[1]{\underline{#1}}
\newcommand{\barA}{\overline A}

\title[Relative Derived Delooping Levels and $\tau$-Tilting Modules]{Relative Derived Delooping Levels and $\tau$-Tilting Modules}
\author[H. Gao]{Hanpeng Gao}
\address{School of Mathematical Sciences, Anhui University, Hefei 230601, Anhui, P. R. China}
\email{hpgao@ahu.edu.cn}

\author[D. Liu]{Dajun Liu$^{*}$}
\address{School of Mathematics-Physics and Finance,
	Anhui Polytechnic University,
	Wuhu 241000, Anhui, P. R. China}
\email{liudajun@ahpu.edu.cn}

\author[H. Zhang]{Houjun Zhang}
\address{School of Science,
	Nanjing University of Posts and Telecommunications,
	Nanjing 210023, Jiangsu, P. R. China}
\email{zhanghoujun@njupt.edu.cn}

\thanks{*Corresponding author}

\date{}

\begin{document}

\begin{abstract}
	
Let $\mathcal{E}$ be an essentially small idempotent-complete exact category with enough projective objects. We develop derived delooping levels  in $\mathcal{E}$. We apply this framework to a basic support $\tau$-tilting module $T$. We then define the derived delooping level relative to $T$ and compare it with the ordinary relative delooping level and the derived delooping level of $B=\operatorname{End}_A(T)$. We show that the big finitistic dimension of $B^{\mathrm{op}}$ is bounded above by the derived delooping level of $\operatorname{Fac}T$ relative to $T$. Moreover, we prove that a self-orthogonal $\tau$-tilting module $T$ is a $1$-tilting module whenever $2\text{-}\ddell_B(DT)$ is finite.	 
\end{abstract}

\maketitle



\section{Introduction}\label{sec:introduction}

The finitistic dimension conjecture is one of the central open problems in the representation theory of Artin algebras. For a finite-dimensional algebra $A$, it asks whether \[ \operatorname{findim}A = \sup\{ \operatorname{pd}_{A}M \mid M\in\operatorname{mod}A,\ \operatorname{pd}_{A}M<\infty \} \] is finite. G\'elinas \cite{Gelinas} introduced the delooping level as a syzygetic invariant related to this problem and showed that it gives an upper bound for the opposite big finitistic dimension. Guo and Igusa \cite{GuoIgusa} later introduced the derived delooping level and proved, for every finite-dimensional algebra $A$, \[ \operatorname{Findim}A^{\mathrm{op}} = \operatorname{edell}A \leq \ddell A \leq \dell A. \] The distinction between the ordinary and derived invariants is genuine: there are finite-dimensional algebras for which $\ddell A$ is finite whereas $\dell A$ is infinite \cite{GuoSymmetry}. The aim of this paper is to develop a relative version of derived delooping theory associated with support $\tau$-tilting modules and to apply it to the homological properties of their endomorphism algebras. Let  $T$ be a basic support $\tau$-tilting right $A$-module, $ B=\operatorname{End}_{A}(T)$, and  $\mathcal C=\operatorname{Fac}T.$ The category $\mathcal C$, equipped with the exact structure induced from $\operatorname{mod}A$, has enough projective objects.  Relative homological invariants of $\operatorname{Fac}T$ have recently been used to study both the finitistic dimension of $B^{\mathrm{op}}$, the depth and the ordinary relative delooping level associated with $T$ provide bounds for the finitistic dimension of $B^{\mathrm{op}}$, see \cite{XuZhang}.  These results naturally lead to the question whether ordinary relative delooping can be replaced by relative derived  delooping.  

We first develop derived delooping levels over exact categories. Let $\mathcal E$ be an essentially small idempotent-complete exact category with enough projective objects. We prove, among other things, that for every conflation \[ 0\longrightarrow X\longrightarrow Y\longrightarrow Z \longrightarrow 0 \] and every $k\geq 1$, \[ k\text{-}\operatorname{ddell}_{\mathcal E}Y \leq k\text{-}\operatorname{ddell}_{\mathcal E}X + k\text{-}\operatorname{ddell}_{\mathcal E}Z +1. \] 
If $X\rightarrowtail Y$ is an inflation, then $\ddell_{\E}X\leq\ddell_{\E}Y+1$. Consequently, the objects of finite derived delooping level form an admissibly torsion-free resolving exact subcategory with enough projective objects.
We also prove a transfer result for exact functors of finite projective amplitude. In particular, exact equivalences preserve ordinary and derived delooping levels. 

We also define a canonical relative derived delooping level $ k\text{-}\operatorname{ddell}_{T}\mathcal C$ for $\mathcal C=\operatorname{Fac}T$. Its definition is obtained from the canonical semibrick of $\mathcal C$ and the boundary objects corresponding, under the Brenner--Butler equivalence, to the radicals of the projective covers of those simple $B$-modules which do not belong to $\operatorname{Sub}(DT)$. This construction is the derived analogue of the ordinary relative delooping level. Our first main result compares this invariant with the homological invariants of the endomorphism algebra. 

\medskip \noindent \textbf{Theorem A.}(\cref{thm:relative-main}) \textit{ Let $T$ be a basic support $\tau$-tilting $A$-module, $B=\operatorname{End}_{A}(T)$ and $\mathcal C=\operatorname{Fac}T$. Then, for every $k\geq 1$, \[ k\text{-}\operatorname{ddell}B \leq k\text{-}\operatorname{ddell}_{T}\mathcal C \leq k\text{-}\operatorname{dell}_{T}\mathcal C. \] In particular, $$ \operatorname{depth}_{T}\mathcal C \leq \operatorname{findim}B^{\mathrm{op}} \leq \operatorname{Findim}B^{\mathrm{op}} \leq \operatorname{ddell}B \leq \operatorname{ddell}_{T}\mathcal C \leq \operatorname{dell}_{T}\mathcal C.$$ 
} 

Thus the canonical relative derived invariant gives a refinement of the ordinary relative upper bound for the big finitistic dimension of $B^{\mathrm{op}}$. 

The quotient $ \overline A=A/\operatorname{ann}_{A}T $also plays a distinguished role.  We prove that the objectwise relative derived delooping levels are controlled, up to one, by this single object (see \cref{thm:DbarA-control}): \[ \operatorname{ddell}_{\mathcal C}(D\overline A) \leq \operatorname{gddell}(\mathcal C) \leq \operatorname{ddell}_{\mathcal C}(D\overline A)+1. \] 

 Let $\mathcal X$ be a resolving exact subcategory of an exact category $\mathcal A$ having the same projective objects. Assume that there exists $c\geq 0$ such that $\Omega_{\mathcal A}^{c}M\in\mathcal X $ for every $M\in\mathcal A$. We prove \[ k\text{-}\operatorname{ddell}_{\mathcal A}X \leq k\text{-}\operatorname{ddell}_{\mathcal X}X \leq c+(k+c)\text{-}\operatorname{ddell}_{\mathcal A}X \] for all $X\in\mathcal X$ and $k\geq 1$ (see \cref{thm:localization}). For $ \mathcal X=\operatorname{Sub}(DT) \subseteq\operatorname{mod}B, $ one may take $c=1$. Consequently, \[ k\text{-}\operatorname{ddell}_{B}X \leq k\text{-}\operatorname{ddell}_{\operatorname{Sub}(DT)}X \leq 1+(k+1)\text{-}\operatorname{ddell}_{B}X \] for every $X\in\operatorname{Sub}(DT)$ (see \cref{cor:SubDT-localization}). This result allows us to apply the relative theory to the problem of deciding when a $\tau$-tilting module is a classical tilting module. Let $T$ be a $\tau$-tilting $A$-module and $ I=\operatorname{ann}_{A}T$.   If $ m=k\text{-}\operatorname{ddell}_{\operatorname{Fac}T}(D\overline A) <\infty$  and $ \operatorname{Ext}_{A}^{j}(T,T)=0 ~ (2\leq j\leq m+k+1)$, then $ \operatorname{Tor}^{A}_{k}(T,I)=0. $ This implies that $T$ is a $1$-tilting $A$-module when  $k=1$ (see \cref{thm:finite-range-tor}). Combining this result, we have the following criterion in terms of the endomorphism algebra.
 
  \medskip \noindent \textbf{Theorem B.}(\cref{thm:ambient-tilting}) \textit{Let $T$ be a $\tau$-tilting $A$-module and $B=\operatorname{End}_{A}(T)$. If \[ d=2\text{-}\operatorname{ddell}_{B}(DT)<\infty \] and \[ \operatorname{Ext}_{A}^{j}(T,T)=0 \qquad (2\leq j\leq d+3), \] then $T$ is a $1$-tilting $A$-module. 
  	In particular, if $T$ is self-orthogonal and $2\text{-}\operatorname{ddell}_{B}(DT)<\infty, $, then $T$ is $1$-tilting. } 
  
  \medskip  The paper is organized as follows. \Cref{sec:preliminaries}  contains the necessary preliminaries. In \Cref{sec:calculus}, we develop derived delooping levels over exact categories and prove the transfer and Ext-detection results. \Cref{sec:relative}  is devoted to the relative derived theory associated with support $\tau$-tilting modules, including the finitistic-dimension comparison, the cogenerator estimate,  and the tilting criteria. \Cref{sec:examples} contains the examples.

\section{Preliminaries}\label{sec:preliminaries}

Throughout the paper, exact categories are assumed to be essentially small and idempotent complete. 
In the representation-theoretic part, $A$ is a finite-dimensional basic algebra over an algebraically closed field $K$, all modules are finitely generated right modules, and
\[
 D=\Hom_K(-,K)
\]
is the ordinary duality. The symbols $\add M$, $\Fac M$ and $\Sub M$ have their usual meanings. All occurrences of $D(A/I)$ mean the dual of the left regular $A/I$-module, viewed as a right module. Left-module formulations are obtained by passage to opposite algebras.

\begin{definition}\label{def:stable-category}
{\rm(\cite{Keller})}
Let $\E$ be an exact category and $\Pcat=\Proj\E$. Its projectively stable category is
\[
 \stable{\E}=\E/[\Pcat].
\]
For every $X\in\E$, fix a projective conflation
\[
 0\longrightarrow \Omega_{\E}X\longrightarrow P_X\longrightarrow X\longrightarrow0.
\]
The induced endofunctor of $\stable{\E}$ is denoted by $\Omega_{\E}$.
\end{definition}

\begin{lemma}\label{lem:schanuel}
Let
\[
 0\longrightarrow K\xrightarrow{i}P\xrightarrow{p}X\longrightarrow0,
 \qquad
 0\longrightarrow K'\xrightarrow{i'}P'\xrightarrow{p'}X\longrightarrow0
\]
be conflations with $P$ and $P'$ projective. Then
\[
 K\oplus P'\simeq K'\oplus P.
\]
Consequently, any two $n$-th syzygies of $X$ become isomorphic after adjoining projective direct summands.
\end{lemma}

\begin{proof}
Let $W=P\times_XP'$ and denote the projections by $q:W\to P$ and $q':W\to P'$. Since deflations are stable under pullback, the right-hand square in
\[
\begin{tikzcd}[column sep=3.2em,row sep=2.5em]
0\arrow[r]&K'\arrow[r,"u"]\arrow[d,equal]&W\arrow[r,"q"]\arrow[d,"q'"]&P\arrow[r]\arrow[d,"p"]&0\\
0\arrow[r]&K'\arrow[r,"i'"]&P'\arrow[r,"p'"]&X\arrow[r]&0
\end{tikzcd}
\]
is cartesian and the upper row is a conflation. The universal property of the pullback identifies $u$ with the unique morphism satisfying $q'u=i'$ and $qu=0$. Since $P$ is projective, the deflation $q$ admits a section; hence the upper row splits and
\begin{equation}\label{eq:schanuel-first-splitting}
 W\simeq K'\oplus P.
\end{equation}
Interchanging the two projective presentations gives a conflation
\[
 0\longrightarrow K\longrightarrow W\xrightarrow{q'}P'\longrightarrow0,
\]
which splits because $P'$ is projective. Thus
\begin{equation}\label{eq:schanuel-second-splitting}
 W\simeq K\oplus P'.
\end{equation}
Combining \eqref{eq:schanuel-first-splitting} and \eqref{eq:schanuel-second-splitting} proves the first assertion.

For the higher assertion, compare the terminal projective deflations in two $n$-step projective resolutions of $X$. The first assertion gives a stable isomorphism between their kernels. Repeating the argument at the preceding stages yields projective objects $Q,Q'$ such that
\[
 \Omega_{\boldsymbol P}^{n}X\oplus Q\simeq
 \Omega_{\boldsymbol P'}^{n}X\oplus Q'.
\]
Therefore the stable isomorphism class of $\Omega_{\E}^{n}X$ is independent of all choices.
\end{proof}

\begin{definition}\label{def:stable-retract}
An object $X\in\E$ is a \emph{stable retract} of $Y\in\E$ if there are morphisms
\[
 X\xrightarrow{u}Y\xrightarrow{v}X
\]
in $\stable{\E}$ with $vu=1_X$.
\end{definition}

\begin{definition}\label{def:dell}
Let $X\in\E$ and $k\geq1$. The $k$-delooping level of $X$ is
\begin{align*}
 k\text{-}\dell_{\E}X
 =\inf\bigl\{n\geq0\mid {}&
 \Omega_{\E}^{n}X\text{ is a stable retract of}\\
 &\Omega_{\E}^{n+k}Y\text{ for some }Y\in\E\bigr\}.
\end{align*}
For $k=1$ we write $\dell_{\E}X$.
\end{definition}

A finite sequence
\[
 0\longrightarrow C_t\longrightarrow C_{t-1}\longrightarrow\cdots
 \longrightarrow C_0\longrightarrow X\longrightarrow0
\]
is called \emph{admissibly exact} if there are objects $Z_0=X,Z_1,\ldots,Z_t=C_t$ and conflations
\[
 0\longrightarrow Z_{i+1}\longrightarrow C_i\longrightarrow Z_i\longrightarrow0
 \qquad(0\leq i<t).
\]
For $t=0$, the map $C_0\to X$ is required to be an isomorphism.

\begin{definition}\label{def:ddell}
Let $X\in\E$ and $k\geq1$. The $k$-derived delooping level of $X$ is
\begin{align*}
 k\text{-}\ddell_{\E}X
 =\inf\bigl\{m\geq0\mid {}&\text{there are }0\leq t\leq m\text{ and an admissibly exact sequence}\\
 &0\to C_t\to\cdots\to C_0\to X\to0,\\
 &(i+k)\text{-}\dell_{\E}C_i\leq m-i\text{ for }0\leq i\leq t\bigr\}.
\end{align*}
For $k=1$ we write $\ddell_{\E}X$.
\end{definition}

When $\E=\modu\Lambda$, this is the derived delooping level of Guo and Igusa \cite{GuoIgusa}. Throughout the paper, both ordinary and derived invariants use the normalization $n\geq0$ in \cref{def:dell}; in particular, projective objects have every ordinary $k$-delooping level equal to zero.

For an Artin algebra $\Lambda$ and $k\geq1$, set
\[
 k\text{-}\dell\Lambda
 =\sup\{k\text{-}\dell_{\Lambda}S\mid S\text{ simple in }\modu\Lambda\},
\]
\[
 k\text{-}\ddell\Lambda
 =\sup\{k\text{-}\ddell_{\Lambda}S\mid S\text{ simple in }\modu\Lambda\}.
\]
We also write
\[
 \findim\Lambda=\sup\{\pd_{\Lambda}M\mid M\in\modu\Lambda,
 \ \pd_{\Lambda}M<\infty\},
\]
\[
 \Findim\Lambda=\sup\{\pd_{\Lambda}M\mid M\in\Mod\Lambda,
 \ \pd_{\Lambda}M<\infty\}.
\]
For an object $X$ of an exact category, $\pd_{\E}X$ is the least $d$ such that a $d$-th syzygy is projective, and is $\infty$ if no such $d$ exists.

\begin{lemma}\label{lem:monotonicity}
	For every $M\in\E$ and $1\leq r\leq s$,
	\[
	r\text{-}\dell_{\E}M\leq s\text{-}\dell_{\E}M,
	\qquad
	r\text{-}\ddell_{\E}M\leq s\text{-}\ddell_{\E}M.
	\]
\end{lemma}

\begin{proof}
	If $\Omega_{\E}^{n}M$ is a stable retract of $\Omega_{\E}^{n+s}N$, then
	\[
	\Omega_{\E}^{n+s}N\simeq
	\Omega_{\E}^{n+r}(\Omega_{\E}^{s-r}N)
	\]
	in $\stable{\E}$, proving the  inequality about delooping leveling.  Applying  to a deriveddelooping leveling, the second inequality can be obtained.
\end{proof}

\begin{proposition}\label{prop:basic-properties}
Let $X,Y\in\E$, let $P,Q\in\Proj\E$, and let $k\geq1$.
\begin{enumerate}
\item $k\text{-}\ddell_{\E}X\leq k\text{-}\dell_{\E}X$.
\item If $X\oplus P\simeq Y\oplus Q$, then
\[
 k\text{-}\dell_{\E}X=k\text{-}\dell_{\E}Y,
 \qquad
 k\text{-}\ddell_{\E}X=k\text{-}\ddell_{\E}Y.
\]
\item If $(k+1)\text{-}\ddell_{\E}X\leq m$, then
$k\text{-}\ddell_{\E}X\leq m$.
\end{enumerate}
\end{proposition}

\begin{proof}
(1) Let $n=k\text{-}\dell_{\E}X<\infty$. If $n=0$, then the admissibly exact sequence with $t=0$ and $C_0=X$. So the result is trivial. If $n>0$, truncate a projective resolution after $n$ steps:
\[
 0\longrightarrow\Omega_{\E}^{n}X\longrightarrow P_{n-1}
 \longrightarrow\cdots\longrightarrow P_0\longrightarrow X\longrightarrow0.
\]
The projective terms have  delooping level zero. Moreover,
\[
 (n+k)\text{-}\dell_{\E}(\Omega_{\E}^{n}X)=0
\]
by the definition. Hence  $k\text{-}\dell_{\E}X\leq
n$.

(2) Delooping level is stable by definition.  The stable isomorphism in the statement then gives equality.

(3) follows from the monotonicity established in \cref{lem:monotonicity}.
\end{proof}

\begin{lemma}\label{lem:syzygy-shift}
Let
\[
 \xi:\quad 0\longrightarrow X\xrightarrow{x}Y\xrightarrow{y}Z\longrightarrow0
\]
be a conflation, and let
\[
 0\longrightarrow\Omega_{\E}Z\xrightarrow{j}P_Z\xrightarrow{\pi}Z\longrightarrow0
\]
be a projective conflation. Then there is a conflation
\[
 0\longrightarrow\Omega_{\E}Z\longrightarrow X\oplus P_Z
 \longrightarrow Y\longrightarrow0.
\]
\end{lemma}

\begin{proof}
 By stability of deflations under pullback, there is a commutative diagram
\[
\begin{tikzcd}[column sep=3.2em,row sep=2.5em]
0\arrow[r]&X\arrow[r,"\widetilde x"]\arrow[d,equal]&W\arrow[r,"q"]\arrow[d,"p"]&P_Z\arrow[r]\arrow[d,"\pi"]&0\\
0\arrow[r]&X\arrow[r,"x"]&Y\arrow[r,"y"]&Z\arrow[r]&0.
\end{tikzcd}
\]
The upper row is a conflation. Since $P_Z$ is projective, $q$ has a section and there is an isomorphism
\begin{equation}\label{eq:pullback-splitting}
 \theta:X\oplus P_Z\xrightarrow{\sim}W
\end{equation}
whose restriction to $X$ is $\widetilde x$.

The equality $yp=\pi q$ and the relation $\pi j=0$ yield, by the pullback universal property, a unique morphism $\widetilde j:\Omega_{\E}Z\to W$ satisfying
$q\widetilde j=j$ and $p\widetilde j=0$. The pullback axiom applied to $\pi$ also shows that
\[
 0\longrightarrow\Omega_{\E}Z\xrightarrow{\widetilde j}W
 \xrightarrow{p}Y\longrightarrow0
\]
is a conflation. Transporting it along \eqref{eq:pullback-splitting} gives the required conflation.
\end{proof}

\begin{lemma}\label{lem:ddell-syzygy-step}
For every $M\in\E$ and $k\geq1$,
\[
 k\text{-}\ddell_{\E}M
 \leq 1+(k+1)\text{-}\ddell_{\E}(\Omega_{\E}M).
\]
\end{lemma}

\begin{proof}
Put $m=(k+1)\text{-}\ddell_{\E}(\Omega_{\E}M)<\infty$ and choose an admissibly exact sequence
\[
 0\to C_t\to\cdots\to C_0\to\Omega_{\E}M\to0,
 \qquad t\leq m,
\]
with $(i+k+1)\text{-}\dell_{\E}C_i\leq m-i$. Splice it with
\[
 0\to\Omega_{\E}M\to P_M\to M\to0.
\]
The resulting admissibly exact sequence has $P_M$ in position zero and $C_i$ in position $i+1$, and is a $k$-derived delooping level of value $m+1$.
\end{proof}

\section{Derived delooping level of exact categories}\label{sec:calculus}

The module-theoretic about  derived delooping levels is due to Guo and Igusa \cite[Lemma~3.1]{GuoIgusa}. Considering the derived delooping levels over exact categories, the point here is that every construction is carried out inside an arbitrary exact structure. In particular, successive syzygy terms must be made compatible by admissible conflations rather than by appealing to   kernels in an abelian category.

\begin{definition}\label{def:admissibly-torsion-free}
A full subcategory $\mathcal F\subseteq\E$ is called \emph{admissibly torsion-free} if it is closed under isomorphisms, direct summands, extensions and admissible subobjects.
\end{definition}

If $\E$ is a length abelian category with its standard exact structure, this is the usual finite-length notion of a torsion-free class.

\begin{lemma}\label{lem:lift-witness}
Let
\begin{equation*}\label{eq:admissible-sequence}
 0\longrightarrow C_t\longrightarrow C_{t-1}\longrightarrow\cdots
 \longrightarrow C_0\longrightarrow X\longrightarrow0
\end{equation*}
be admissibly exact in $\E$, and let $s\geq0$. There is an admissibly exact sequence
\begin{equation}\label{eq:lifted-sequence}
 0\longrightarrow C_t^{[s]}\longrightarrow C_{t-1}^{[s]}\longrightarrow\cdots
 \longrightarrow C_0^{[s]}\longrightarrow\Omega_{\E}^{s}X\longrightarrow0
\end{equation}
such that
\begin{equation}\label{eq:stable-terms}
 C_i^{[s]}\simeq\Omega_{\E}^{s}C_i
 \qquad\text{in }\stable{\E}
\end{equation}
for $0\leq i\leq t$.
\end{lemma}

\begin{proof}
The statement is trivial for $s=0$ or $t=0$.    Hence assume $s,t\geq1$.

Choose objects $Z_0=X,Z_1,\ldots,Z_t=C_t$ and conflations
\begin{equation*}\label{eq:cycles}
 \xi_i:\quad
 0\longrightarrow Z_{i+1}\xrightarrow{a_i}C_i
 \xrightarrow{b_i}Z_i\longrightarrow0
 \qquad(0\leq i<t).
\end{equation*}
We first construct \eqref{eq:lifted-sequence} for $s=1$. For every $i$, fix a projective conflation
\[
 0\longrightarrow K_i\xrightarrow{u_i}P_i\xrightarrow{v_i}Z_i\longrightarrow0,
 \qquad K_i=\Omega_{\E}Z_i.
\]
Apply the exact-category Horseshoe Lemma \cite[Theorem~12.8]{Buhler} to $\xi_i$ and to the fixed projective presentations of its end terms. It yields a commutative diagram with conflation rows and columns
\begin{equation}\label{eq:horseshoe-diagram}
\begin{tikzcd}[column sep=2.8em,row sep=2.4em]
0\arrow[r]&K_{i+1}\arrow[r]\arrow[d,"u_{i+1}"]&H_i\arrow[r]\arrow[d]&K_i\arrow[r]\arrow[d,"u_i"]&0\\
0\arrow[r]&P_{i+1}\arrow[r]\arrow[d,"v_{i+1}"]&P_{i+1}\oplus P_i\arrow[r]\arrow[d]&P_i\arrow[r]\arrow[d,"v_i"]&0\\
0\arrow[r]&Z_{i+1}\arrow[r,"a_i"]&C_i\arrow[r,"b_i"]&Z_i\arrow[r]&0.
\end{tikzcd}
\end{equation}
The middle column is a projective deflation onto $C_i$ and has kernel $H_i$. By \cref{lem:schanuel},
\begin{equation*}\label{eq:horseshoe-term-stable}
 H_i\simeq\Omega_{\E}C_i
 \qquad\text{in }\stable{\E}.
\end{equation*}
Set $C_i^{[1]}=H_i$ for $0\leq i<t$ and $C_t^{[1]}=K_t=\Omega_{\E}C_t$. The upper rows in \eqref{eq:horseshoe-diagram} have common cycle objects: the right-hand term for $i+1$ and the left-hand term for $i$ are both the fixed object $K_{i+1}$. They therefore splice to
\[
 0\longrightarrow C_t^{[1]}\longrightarrow C_{t-1}^{[1]}
 \longrightarrow\cdots\longrightarrow C_0^{[1]}
 \longrightarrow K_0=\Omega_{\E}X\longrightarrow0.
\]
This sequence is admissibly exact and satisfies \eqref{eq:stable-terms} for $s=1$.

Assume inductively that \eqref{eq:lifted-sequence} has been constructed for $s$. Apply the case $s=1$ to that sequence, keeping one fixed projective presentation for each of its cycle objects. The resulting term at position $i$ is stably isomorphic to $\Omega_{\E}C_i^{[s]}$. Applying the syzygy functor to \eqref{eq:stable-terms} and using \cref{lem:schanuel} gives
\[
 \Omega_{\E}C_i^{[s]}
 \simeq\Omega_{\E}^{s+1}C_i
 \qquad\text{in }\stable{\E}.
\]
This proves the assertion by induction on $s$.
\end{proof}

\begin{lemma}\label{lem:splice}
Let
\begin{equation*}\label{eq:conflation-splice}
 0\longrightarrow X\longrightarrow Y\longrightarrow Z\longrightarrow0
\end{equation*}
be a conflation, and let
\begin{equation}\label{eq:witness-X}
 0\longrightarrow D_p\longrightarrow D_{p-1}\longrightarrow\cdots
 \longrightarrow D_0\longrightarrow X\longrightarrow0
\end{equation}
be admissibly exact. If
\begin{equation}\label{eq:resolution-Z}
 0\longrightarrow\Omega_{\E}^{p+1}Z\longrightarrow P_p\longrightarrow\cdots
 \longrightarrow P_0\longrightarrow Z\longrightarrow0
\end{equation}
is  a projective resolution, then there is an admissibly exact sequence
\begin{equation}\label{eq:spliced-Y}
 0\longrightarrow\Omega_{\E}^{p+1}Z\longrightarrow D_p\oplus P_p
 \longrightarrow\cdots\longrightarrow D_0\oplus P_0
 \longrightarrow Y\longrightarrow0.
\end{equation}
\end{lemma}

\begin{proof}
Choose cycle objects $X_0=X,X_1,\ldots,X_p=D_p$ for \eqref{eq:witness-X}; thus
\begin{equation*}\label{eq:X-cycle-conflations}
 0\longrightarrow X_{i+1}\xrightarrow{a_i}D_i
 \xrightarrow{d_i}X_i\longrightarrow0
 \qquad(0\leq i<p).
\end{equation*}
Put $X_{p+1}=0$ and append the split conflation
$0\to0\to D_p\xrightarrow{1}X_p\to0$. Write $Z_i=\Omega_{\E}^{i}Z$ and decompose \eqref{eq:resolution-Z} as
\begin{equation*}\label{eq:Z-cycle-conflations}
 0\longrightarrow Z_{i+1}\xrightarrow{u_i}P_i
 \xrightarrow{v_i}Z_i\longrightarrow0
 \qquad(0\leq i\leq p).
\end{equation*}
We construct, by induction on $i$, conflations
\begin{equation*}\label{eq:Yi}
 \eta_i:\quad0\longrightarrow X_i\longrightarrow Y_i
 \xrightarrow{y_i}Z_i\longrightarrow0
 \qquad(0\leq i\leq p+1)
\end{equation*}
with $Y_0=Y$, together with deflations
\begin{equation*}\label{eq:DiPi-deflation}
 h_i:D_i\oplus P_i\twoheadrightarrow Y_i
 \quad\text{and}\quad\Ker h_i=Y_{i+1}.
\end{equation*}

Assume that $\eta_i$ has been constructed.
Considering the pullback diagram with conflation rows:
\begin{equation*}\label{eq:splice-pullback}
\begin{tikzcd}[column sep=3em,row sep=2.4em]
0\arrow[r]&X_i\arrow[r]\arrow[d,equal]&W_i\arrow[r,"q_i"]\arrow[d,"g_i"]&P_i\arrow[r]\arrow[d,"v_i"]&0\\
0\arrow[r]&X_i\arrow[r]&Y_i\arrow[r,"y_i"]&Z_i\arrow[r]&0.
\end{tikzcd}
\end{equation*}
 The upper row splits because $P_i$ is projective. Fix an isomorphism
$\sigma_i:X_i\oplus P_i\xrightarrow{\sim}W_i$, let
\[
 f_i=\sigma_i(d_i\oplus1_{P_i}):D_i\oplus P_i\twoheadrightarrow W_i.
\]
The direct sum $d_i\oplus1_{P_i}$ is a deflation with kernel $X_{i+1}$; hence so is $f_i$. The morphism $g_i$ is a deflation with kernel $Z_{i+1}$ by the pullback property. Thus the composite
$h_i=g_if_i$ is a deflation. Define $Y_{i+1}=\Ker h_i$.

The inclusions of the three kernels are compatible with $f_i$ and $g_i$. The universal property of $\Ker h_i$ gives a unique morphism $X_{i+1}\to Y_{i+1}$, and the universal property of $\Ker g_i$ gives a unique morphism $Y_{i+1}\to Z_{i+1}$. The $3\times3$ Lemma \cite[Lemma~3.6]{Buhler}, applied to the kernel diagram of the composable deflations $f_i$ and $g_i$, yields the conflation
\begin{equation}\label{eq:kernel-composite-conflation}
 0\longrightarrow X_{i+1}\longrightarrow Y_{i+1}
 \longrightarrow Z_{i+1}\longrightarrow0.
\end{equation}
This is $\eta_{i+1}$ and completes the induction.

Note that $X_{p+1}=0$, we have
$Y_{p+1}\simeq Z_{p+1}=\Omega_{\E}^{p+1}Z$. Finally, we obtain \eqref{eq:spliced-Y}.
\end{proof}

\begin{theorem}\label{thm:extension-estimate}
Let
\[
 0\longrightarrow X\longrightarrow Y\longrightarrow Z\longrightarrow0
\]
be a conflation in $\E$. Then, for every $k\geq1$,
\begin{equation}\label{eq:extension-estimate}
 k\text{-}\ddell_{\E}Y
 \leq k\text{-}\ddell_{\E}X+k\text{-}\ddell_{\E}Z+1.
\end{equation}
\end{theorem}

\begin{proof}
The assertion is immediate if one of the two terms on the right-hand side is infinite. Set
\[
 m_1=k\text{-}\ddell_{\E}X,
 \qquad m_2=k\text{-}\ddell_{\E}Z,
 \qquad M=m_1+m_2+1.
\]
Choose admissibly exact witnesses
\begin{equation}\label{eq:witness1}
 0\longrightarrow D_p\longrightarrow\cdots\longrightarrow D_0
 \longrightarrow X\longrightarrow0,
 \qquad p\leq m_1,
\end{equation}
\begin{equation}\label{eq:witness2}
 0\longrightarrow E_q\longrightarrow\cdots\longrightarrow E_0
 \longrightarrow Z\longrightarrow0,
 \qquad q\leq m_2,
\end{equation}
satisfying
\begin{equation}\label{eq:witness-bounds}
 (i+k)\text{-}\dell_{\E}D_i\leq m_1-i,
 \qquad
 (j+k)\text{-}\dell_{\E}E_j\leq m_2-j.
\end{equation}
By \cref{lem:splice}, \eqref{eq:witness1} and the first $p+1$ steps of a projective resolution of $Z$ give
\begin{equation}\label{eq:Y-initial}
 0\longrightarrow\Omega_{\E}^{p+1}Z\longrightarrow D_p\oplus P_p
 \longrightarrow\cdots\longrightarrow D_0\oplus P_0
 \longrightarrow Y\longrightarrow0.
\end{equation}
Applying \cref{lem:lift-witness} with $s=p+1$ to \eqref{eq:witness2} gives
\begin{equation}\label{eq:Z-lifted}
 0\longrightarrow E_q^{[p+1]}\longrightarrow\cdots
 \longrightarrow E_0^{[p+1]}\longrightarrow\Omega_{\E}^{p+1}Z
 \longrightarrow0,
\end{equation}
where $E_j^{[p+1]}\simeq\Omega_{\E}^{p+1}E_j$ in $\stable{\E}$. Splicing \eqref{eq:Z-lifted} with \eqref{eq:Y-initial} produces a sequence of length $p+q+1\leq M$. Its term in position $r$ is
\[
 W_r=
 \begin{cases}
 D_r\oplus P_r,&0\leq r\leq p,\\
 E_{r-p-1}^{[p+1]},&p+1\leq r\leq p+q+1.
 \end{cases}
\]
We verify the defining estimate
\begin{equation}\label{eq:Wr-bound}
 (r+k)\text{-}\dell_{\E}W_r\leq M-r.
\end{equation}

For $0\leq r\leq p$, projective stability and \eqref{eq:witness-bounds} give
\[
 (r+k)\text{-}\dell_{\E}W_r
 =(r+k)\text{-}\dell_{\E}D_r
 \leq m_1-r\leq M-r.
\]
Let $r=p+1+j$ with $0\leq j\leq q$, and set
$a_j=(j+k)\text{-}\dell_{\E}E_j$. Choose a stable retraction
\begin{equation}\label{eq:Ej-retraction}
 \Omega_{\E}^{a_j}E_j
 \mathrel{\substack{\longrightarrow\\[-2pt]\longleftarrow}}
 \Omega_{\E}^{a_j+j+k}N_j
\end{equation}
whose composite on the left is the identity. Since $a_j\leq m_2-j$, the integer
$s_j=m_1+m_2-j+1-a_j$ is nonnegative. Applying $\Omega_{\E}^{s_j}$ to \eqref{eq:Ej-retraction} yields
\[
 \Omega_{\E}^{m_1+m_2-j+1}E_j
 \mathrel{\substack{\longrightarrow\\[-2pt]\longleftarrow}}
 \Omega_{\E}^{m_1+m_2+k+1}N_j.
\]
Put $\ell_j=m_1+m_2-p-j$. Then $\ell_j\geq0$ and
\[
 \ell_j+p+1=m_1+m_2-j+1,
 \qquad
 \ell_j+(p+1+j+k)=m_1+m_2+k+1.
\]
Consequently,
\[
 (p+1+j+k)\text{-}\dell_{\E}(\Omega_{\E}^{p+1}E_j)
 \leq\ell_j=M-(p+1+j).
\]
By \cref{prop:basic-properties}(2), the same estimate holds for
$W_{p+1+j}=E_j^{[p+1]}$. This proves \eqref{eq:Wr-bound} in all positions. Hence the spliced sequence is a $k$-derived witness of value $M$ for $Y$, proving \eqref{eq:extension-estimate}.
\end{proof}

\begin{remark}\label{rem:extension-literature}
For $\E=\modu\Lambda$, \cref{thm:extension-estimate} specializes to \cite[Lemma~3.1]{GuoIgusa}. The role of \cref{lem:lift-witness,lem:splice} is to replace those operations by admissible Horseshoe and pullback constructions inside an arbitrary exact structure.
\end{remark}

\begin{corollary}\label{cor:subobject}
If $X\rightarrowtail Y$ is an inflation, then
\[
 \ddell_{\E}X\leq\ddell_{\E}Y+1.
\]
\end{corollary}\begin{proof}
Choose a conflation
\begin{equation*}\label{eq:subobject-cokernel}
 0\longrightarrow X\longrightarrow Y\longrightarrow Z\longrightarrow0
\end{equation*}
completing the given inflation. Let
$0\to\Omega_{\E}Z\to P_Z\to Z\to0$ be a projective conflation. The pullback construction in \cref{lem:syzygy-shift} gives
\begin{equation*}\label{eq:subobject-shifted-conflation}
 0\longrightarrow\Omega_{\E}Z\longrightarrow X\oplus P_Z
 \longrightarrow Y\longrightarrow0.
\end{equation*}
The identity morphism of $\Omega_{\E}Z$ is a stable retraction
\[
 \Omega_{\E}Z
 \mathrel{\substack{\longrightarrow\\[-2pt]\longleftarrow}}
 \Omega_{\E}Z,
\]
and the target is the first syzygy of $Z$. Hence
$\dell_{\E}(\Omega_{\E}Z)=0$ and, by \cref{prop:basic-properties}(1),
$\ddell_{\E}(\Omega_{\E}Z)=0$. Applying \cref{thm:extension-estimate}, we have
\[
 \ddell_{\E}(X\oplus P_Z)
 \leq\ddell_{\E}(\Omega_{\E}Z)+\ddell_{\E}Y+1
 =\ddell_{\E}Y+1.
\]
Finally, projective stability in \cref{prop:basic-properties}(2) gives
$\ddell_{\E}(X\oplus P_Z)=\ddell_{\E}X$.
\end{proof}

\begin{lemma}\label{lem:finite-sums}
Let $X_1,\ldots,X_r\in\E$ and $k\geq1$. Then
\[
 k\text{-}\dell_{\E}\Bigl(\bigoplus_{j=1}^rX_j\Bigr)
 \leq\max_j k\text{-}\dell_{\E}X_j,
\]
\[
 k\text{-}\ddell_{\E}\Bigl(\bigoplus_{j=1}^rX_j\Bigr)
 \leq\max_j k\text{-}\ddell_{\E}X_j.
\]
\end{lemma}

\begin{proof}
Set $a=\max\limits_j k\text{-}\dell_{\E}X_j<\infty$. For every $j$, choose $a_j\leq a$ and a stable retraction
\[
 \Omega_{\E}^{a_j}X_j
 \mathrel{\substack{\longrightarrow\\[-2pt]\longleftarrow}}
 \Omega_{\E}^{a_j+k}Y_j.
\]
After applying $\Omega_{\E}^{a-a_j}$, this becomes a stable retraction at position $a$. Taking the direct sum of the retraction maps gives
\[
 \Omega_{\E}^{a}\Bigl(\bigoplus_jX_j\Bigr)
 \mathrel{\substack{\longrightarrow\\[-2pt]\longleftarrow}}
 \Omega_{\E}^{a+k}\Bigl(\bigoplus_jY_j\Bigr),
\]
which proves the first inequality. Applying the  first inequality to the admissibly exact sequence of $k\text{-}\ddell_{\E}X_j$, we obtain the second
 inequality.
\end{proof}

\begin{theorem}\label{thm:fddell-resolving}
The full subcategory
\[
 \Fddell(\E)=\{X\in\E\mid\ddell_{\E}X<\infty\}
\]
is admissibly torsion-free. With the exact structure induced from $\E$, it is a resolving exact subcategory with enough projective objects. If $\E$ is a length abelian category, $\Fddell(\E)$ is a torsion-free class.
\end{theorem}

\begin{proof}
Every projective object has derived delooping level zero. If
$0\to X\to Y\to Z\to0$ is a conflation with $X,Z\in\Fddell(\E)$, then \cref{thm:extension-estimate} gives $Y\in\Fddell(\E)$. If $X\rightarrowtail Y$ is an inflation and $Y\in\Fddell(\E)$, then \cref{cor:subobject} gives $X\in\Fddell(\E)$. Closure under finite direct sums follows from \cref{lem:finite-sums}; closure under direct summands follows because a split monomorphism is an inflation. Thus $\Fddell(\E)$ is admissibly torsion-free.

Let $X\in\Fddell(\E)$ and choose a projective conflation in $\E$,
\begin{equation}\label{eq:Fddell-projective-conflation}
 0\longrightarrow\Omega_{\E}X\longrightarrow P_X\longrightarrow X\longrightarrow0.
\end{equation}
The object $P_X$ belongs to $\Fddell(\E)$, and $\Omega_{\E}X$ is an admissible subobject of $P_X$; hence $\Omega_{\E}X\in\Fddell(\E)$. Therefore \eqref{eq:Fddell-projective-conflation} is a conflation in the induced exact structure and proves that $\Fddell(\E)$ has enough projective objects. If $P$ is projective in $\E$, then it is projective in the induced exact structure. Conversely, if $Q\in\Fddell(\E)$ is projective there, apply projectivity to \eqref{eq:Fddell-projective-conflation} with $X=Q$; the deflation $P_Q\twoheadrightarrow Q$ splits in $\E$, so $Q$ is projective in $\E$. Hence
\[
 \Proj\Fddell(\E)=\Proj\E.
\]
Finally, for a deflation $Y\twoheadrightarrow Z$ between objects of $\Fddell(\E)$, its kernel is an admissible subobject of $Y$ and therefore belongs to $\Fddell(\E)$. Thus the subcategory is resolving. If $\E$ is length abelian, every monomorphism is an inflation, so admissible torsion-freeness is ordinary torsion-freeness.
\end{proof}

We now prove the functorial transfer theorem. Its abelian counterpart appears in \cite[Proposition~4.4]{WuLiuWei}.

\begin{definition}\label{def:projective-amplitude}
Let $F:\E\to\E'$ be exact, where both exact categories have enough projective objects. Define
\[
 \pd(F)=\sup\{\pd_{\E'}F(P)\mid P\in\Proj\E\}.
\]
\end{definition}

\begin{lemma}\label{lem:dimension-shift-stable}
Let
\[
 0\longrightarrow X\longrightarrow Y\longrightarrow Z\longrightarrow0
\]
be a conflation in an exact category with enough projective objects. If
$\pd Y\leq c<\infty$, then
\[
 \Omega^{c+1}Z\simeq\Omega^{c}X
\]
in the projectively stable category.
\end{lemma}

\begin{proof}
By \cref{lem:syzygy-shift}, there is a conflation
\[
 0\to\Omega Z\to X\oplus P_Z\to Y\to0.
\]
Successive applications of the exact Horseshoe Lemma give, for every $j\geq0$, a conflation
\[
 0\longrightarrow\Omega^{j+1}Z\longrightarrow\Omega^jX\oplus Q_j
 \longrightarrow\Omega^jY\longrightarrow0
\]
with $Q_j$ projective. At $j=c$ the right-hand term is projective, so the conflation splits. The asserted stable isomorphism follows.
\end{proof}

\begin{lemma}\label{lem:stable-functor}
Let $f:X\to Y$ factor through an object $E$ with $\pd E\leq c$. Then every morphism on $c$-th syzygies induced by a comparison lift of $f$ factors through a projective object. Consequently, if $F:\E\to\E'$ is exact and $\pd(F)\leq c$, then, after fixing projective resolutions in $\E'$, the assignment
\[
 X\longmapsto\Omega_{\E'}^{c}F(X)
\]
induces an additive functor
\[
 F_c:\stable{\E}\longrightarrow\stable{\E'}.
\]
Two choices of projective resolutions give naturally isomorphic stable functors.
\end{lemma}

\begin{proof}
If $c=0$, the first assertion says that a morphism factoring through a projective object is zero in the stable category, and the construction is immediate. Assume $c\geq1$. Fix augmented projective resolutions
\[
 \boldsymbol P_X\twoheadrightarrow X,
 \qquad \boldsymbol P_E\twoheadrightarrow E,
 \qquad \boldsymbol P_Y\twoheadrightarrow Y.
\]
Write $f=ba$ with $a:X\to E$ and $b:E\to Y$. By the comparison theorem for exact categories \cite[Section~12]{Buhler}, choose chain maps
$a_\bullet:\boldsymbol P_X\to\boldsymbol P_E$ and
$b_\bullet:\boldsymbol P_E\to\boldsymbol P_Y$ lifting $a$ and $b$. The composite $b_\bullet a_\bullet$ lifts $f$. Its restriction to the $c$-th syzygy is the composite
\[
 \Omega_{\E}^{c}X\longrightarrow\Omega_{\E}^{c}E
 \longrightarrow\Omega_{\E}^{c}Y.
\]
Since $\pd E\leq c$, the middle object is projective; hence the induced morphism is zero in $\stable{\E}$.

We next verify independence of a comparison lift. Let $\varphi_\bullet$ and $\psi_\bullet$ be chain maps between fixed projective resolutions lifting the same morphism. The comparison theorem gives a chain homotopy $h_\bullet$ satisfying
\begin{equation}\label{eq:chain-homotopy-identity}
 \varphi_{c-1}-\psi_{c-1}
 =d_c^Yh_{c-1}+h_{c-2}d_{c-1}^X.
\end{equation}
Restricting \eqref{eq:chain-homotopy-identity} to
$\Omega_{\E}^{c}X=\Ker d_{c-1}^X$ eliminates the second summand. The remaining difference factors as
\[
 \Omega_{\E}^{c}X\xrightarrow{h_{c-1}}P_c^Y
 \xrightarrow{d_c^Y}\Omega_{\E}^{c}Y,
\]
and hence factors through the projective object $P_c^Y$. Thus the stable morphism induced on $c$-th syzygies depends only on the original morphism.

Now fix, for every object $X\in\E$, a projective resolution of $F(X)$ in $\E'$. For a morphism $f:X\to Y$, lift $F(f)$ to a chain map and denote the resulting stable morphism on $c$-th syzygies by $F_c(\stable f)$. If $\stable f=\stable g$, then $f-g$ factors through a projective object $P\in\E$. Consequently $F(f-g)$ factors through $F(P)$, whose projective dimension is at most $c$. The first paragraph, applied in $\E'$, shows that $F_c(\stable f)=F_c(\stable g)$.

Let $f,g:X\to Y$. If $\widetilde f_\bullet$ and $\widetilde g_\bullet$ are comparison lifts of $F(f)$ and $F(g)$, then $\widetilde f_\bullet+\widetilde g_\bullet$ lifts $F(f+g)$. By the preceding independence statement,
\[
 F_c(\stable{f+g})=F_c(\stable f)+F_c(\stable g).
\]
If $X\xrightarrow{f}Y\xrightarrow{g}Z$, then the composite of comparison lifts of $F(f)$ and $F(g)$ is a comparison lift of $F(gf)$; hence
$F_c(\stable{gf})=F_c(\stable g)F_c(\stable f)$. Identity chain maps give identity morphisms. Thus $F_c$ is an additive functor.

Finally, choose a second projective resolution for every $F(X)$. Comparison maps in both directions lifting $1_{F(X)}$ induce mutually inverse stable morphisms
\[
 \theta_X:\Omega_{\E'}^{c}F(X)\xrightarrow{\sim}
 \Omega_{\E'}^{c}F(X)_{\mathrm{new}}.
\]
For $f:X\to Y$, the two chain maps obtained by traversing the naturality square lift the same morphism $F(f)$; they are chain homotopic. Equation \eqref{eq:chain-homotopy-identity} therefore shows that
$\theta_YF_c(\stable f)=F_c^{\rm new}(\stable f)\theta_X$ in $\stable{\E'}$. Hence $\theta$ is a natural isomorphism.
\end{proof}

\begin{proposition}\label{prop:syzygy-comparison}
Let $F:\E\to\E'$ be exact and $\pd(F)\leq c<\infty$. For every $n\geq0$ there is a natural isomorphism in $\stable{\E'}$
\begin{equation}\label{eq:syzygy-comparison}
 \Omega_{\E'}^{c+n}F(-)
 \xrightarrow{\sim}
 \Omega_{\E'}^{c}F(\Omega_{\E}^{n}-).
\end{equation}
Here naturality is understood after the choices of projective resolutions used to define $F_c$ in \cref{lem:stable-functor}.
\end{proposition}

\begin{proof}
The case $n=0$ is the identity. It is enough to construct a natural isomorphism for $n=1$ and then iterate it. Fix, for every $X\in\E$, a projective conflation
\begin{equation}\label{eq:chosen-projective-conflation-transfer}
 0\longrightarrow\Omega_{\E}X\xrightarrow{u_X}P_X
 \xrightarrow{v_X}X\longrightarrow0.
\end{equation}
For a morphism $f:X\to Y$, projectivity of $P_X$ gives a morphism
$p_f:P_X\to P_Y$ satisfying $v_Yp_f=fv_X$. Since
$v_Yp_fu_X=0$, the universal property of the kernel $u_Y$ gives a unique
morphism $\omega_f:\Omega_{\E}X\to\Omega_{\E}Y$ such that
$u_Y\omega_f=p_fu_X$. Hence
\begin{equation}\label{eq:morphism-projective-conflations-transfer}
\begin{tikzcd}[column sep=3em,row sep=2.2em]
0\arrow[r]&\Omega_{\E}X\arrow[r,"u_X"]\arrow[d,"\omega_f"]
 &P_X\arrow[r,"v_X"]\arrow[d,"p_f"]
 &X\arrow[r]\arrow[d,"f"]&0\\
0\arrow[r]&\Omega_{\E}Y\arrow[r,"u_Y"]
 &P_Y\arrow[r,"v_Y"]&Y\arrow[r]&0
\end{tikzcd}
\end{equation}
is a morphism of conflations.

Apply $F$ to \eqref{eq:morphism-projective-conflations-transfer}. For each of its two rows, perform the successive Horseshoe construction used in
\cref{lem:dimension-shift-stable}. The comparison theorem, applied to the vertical maps in \eqref{eq:morphism-projective-conflations-transfer}, yields for every $0\leq j\leq c$ a morphism of conflations
\begin{equation}\label{eq:natural-dimension-shift-diagram}
\begin{tikzcd}[column sep=1.75em,row sep=2.35em,cells={nodes={font=\small}}]
0\arrow[r]&\Omega_{\E'}^{j+1}F(X)\arrow[r]\arrow[d,"\varphi_{j+1}"']
 &\Omega_{\E'}^{j}F(\Omega_{\E}X)\oplus Q_j^X
   \arrow[r]\arrow[d,"\lambda_j"]
 &\Omega_{\E'}^{j}F(P_X)\arrow[r]\arrow[d,"\mu_j"]&0\\
0\arrow[r]&\Omega_{\E'}^{j+1}F(Y)\arrow[r]
 &\Omega_{\E'}^{j}F(\Omega_{\E}Y)\oplus Q_j^Y
   \arrow[r]
 &\Omega_{\E'}^{j}F(P_Y)\arrow[r]&0.
\end{tikzcd}
\end{equation}
where $Q_j^X$ and $Q_j^Y$ are projective. The stable classes of
$\varphi_{j+1}$ and $\mu_j$ are induced by $F(f)$ and $F(p_f)$,
respectively, while the component of $\lambda_j$ between the
nonprojective summands is induced by $F(\omega_f)$. By
\cref{lem:stable-functor}, different comparison choices change the vertical
maps only by morphisms through projective objects.

For $j=c$, the amplitude hypothesis gives
\[
 \pd_{\E'}F(P_X)\leq c,
 \qquad
 \pd_{\E'}F(P_Y)\leq c.
\]
Hence the right-hand terms in \eqref{eq:natural-dimension-shift-diagram}
are projective and both rows split. Let $i_X$ be the first inflation in
the upper row. Its stable class
$\stable{i_X}$ is an isomorphism. Indeed, any retraction of $i_X$ is an
inverse modulo a morphism through the projective cokernel. The projections
\[
 \Omega_{\E'}^{c}F(\Omega_{\E}X)\oplus Q_c^X
 \longrightarrow \Omega_{\E'}^{c}F(\Omega_{\E}X)
\]
and the analogous projection for $Y$ are also isomorphisms in
$\stable{\E'}$. Their composites with $\stable{i_X}$ and
$\stable{i_Y}$ define isomorphisms
\[
 \delta_X:\Omega_{\E'}^{c+1}F(X)
 \xrightarrow{\sim}\Omega_{\E'}^{c}F(\Omega_{\E}X),
 \qquad
 \delta_Y:\Omega_{\E'}^{c+1}F(Y)
 \xrightarrow{\sim}\Omega_{\E'}^{c}F(\Omega_{\E}Y).
\]
The commutativity of \eqref{eq:natural-dimension-shift-diagram} at
$j=c$ gives
\[
 \delta_Y\,\Omega_{\E'}^{c+1}F(\stable f)
 =\Omega_{\E'}^{c}F(\stable{\omega_f})\,\delta_X
 \qquad\text{in }\stable{\E'}.
\]
If $p_f'$ is another lift of $f$, then $v_Y(p_f-p_f')=0$, so the kernel universal property gives $s:P_X\to\Omega_{\E}Y$ with $p_f-p_f'=u_Ys$. Consequently,
\[
 u_Y(\omega_f-\omega_f')=(p_f-p_f')u_X=u_Ysu_X.
\]
Since $u_Y$ is an inflation, $\omega_f-\omega_f'=su_X$; this difference factors through the projective object $P_X$. Hence $\stable{\omega_f}=\stable{\omega_f'}$, and the displayed square depends only on $\stable f$. This proves the natural isomorphism
\[
 \Omega_{\E'}^{c+1}F(-)
 \xrightarrow{\sim}\Omega_{\E'}^{c}F(\Omega_{\E}-).
\]
Composing it successively for
$X,\Omega_{\E}X,\ldots,\Omega_{\E}^{n-1}X$ proves
\eqref{eq:syzygy-comparison} for every $n\geq0$.
\end{proof}

\begin{theorem}\label{thm:transfer}
Let $F:\E\to\E'$ be an exact functor between exact categories with enough projective objects, and let $c=\pd(F)<\infty$. Then, for every $X\in\E$ and every $k\geq1$,
\begin{equation}\label{eq:dell-transfer}
 k\text{-}\dell_{\E'}F(X)
 \leq c+k\text{-}\dell_{\E}X,
\end{equation}
\begin{equation}\label{eq:ddell-transfer1}
 k\text{-}\ddell_{\E'}F(X)
 \leq c+k\text{-}\ddell_{\E}X.
\end{equation}
\end{theorem}

\begin{proof}
Let $n=k\text{-}\dell_{\E}X<\infty$ and choose a stable retraction
\[
 \Omega_{\E}^{n}X\rightleftarrows\Omega_{\E}^{n+k}Y.
\]
Apply the stable functor $F_c$ and use \cref{prop:syzygy-comparison}. This gives a stable retraction
\[
 \Omega_{\E'}^{c+n}F(X)
 \rightleftarrows
 \Omega_{\E'}^{c+n+k}F(Y),
\]
which proves \eqref{eq:dell-transfer}.

Now put $m=k\text{-}\ddell_{\E}X<\infty$ and choose an admissibly exact sequence
\[
 0\to C_t\to\cdots\to C_0\to X\to0,
 \qquad t\leq m,
\]
with $(i+k)\text{-}\dell_{\E}C_i\leq m-i$. Since $F$ is exact, its image is admissibly exact. By \eqref{eq:dell-transfer},
\[
 (i+k)\text{-}\dell_{\E'}F(C_i)
 \leq c+m-i=(c+m)-i.
\]
Moreover $t\leq m\leq c+m$, so \eqref{eq:ddell-transfer1} can be obtained.
\end{proof}

In particular, an exact equivalence and its exact quasi-inverse have projective amplitude zero; hence exact equivalences preserve every ordinary and derived delooping level and identify the corresponding finite-level loci. This assertion concerns exact equivalences of exact categories. It is compatible with Chen's examples for derived equivalences of algebras \cite{ChenDerived}, since a triangle equivalence need not restrict to an exact equivalence of the standard module categories or preserve their syzygy exact structures.

We finish the section with the  homological detection theorem needed for the relative applications below. Let $\Acat$ be an abelian category and let $\E\subseteq\Acat$ be a full extension-closed subcategory equipped with the exact structure induced from $\Acat$. Assume that $\E$ has enough projective objects. All Ext groups below are Yoneda Ext groups in $\Acat$.

\begin{lemma}\label{lem:ambient-dimension-shift}
Let $T\in\Proj\E$, $Y\in\E$, $s\geq1$ and $q\geq0$. Assume
\[
 \Ext_{\Acat}^{j}(T,P)=0
 \qquad(P\in\Proj\E,\ 2\leq j\leq s+q).
\]
Then
\begin{equation}\label{eq:ambient-shift}
 \Ext_{\Acat}^{s}(T,Y)
 \simeq
 \Ext_{\Acat}^{s+q}(T,\Omega_{\E}^{q}Y).
\end{equation}
For $s=1$, both sides are zero.
\end{lemma}

\begin{proof}
If $q=0$, \eqref{eq:ambient-shift} is the identity. Since $\E$ is full and extension closed,  the canonical Yoneda map
\[
 \Ext_{\E}^{1}(T,Y)\longrightarrow\Ext_{\Acat}^{1}(T,Y)
\]
is an isomorphism. Its source is zero because $T$ is projective in $\E$. Thus the assertion for $s=1$ starts with a zero left-hand side.

Let
\[
 0\to\Omega_{\E}Y\to P_0\to Y\to0
\]
be a projective conflation. For $r\geq2$, the ambient long exact sequence gives
\[
 \Ext_{\Acat}^{r}(T,Y)
 \simeq\Ext_{\Acat}^{r+1}(T,\Omega_{\E}Y)
\]
whenever $\Ext_{\Acat}^{r}(T,P_0)$ and $\Ext_{\Acat}^{r+1}(T,P_0)$ vanish. For $r=1$, the group $\Ext_{\Acat}^{1}(T,P_0)$ is zero by the preceding Yoneda argument, while $\Ext_{\Acat}^{2}(T,P_0)=0$ by hypothesis. Iterating along the chosen relative projective resolution gives \eqref{eq:ambient-shift}. In the case $s=1$, the same iteration shows that the right-hand side is zero.
\end{proof}

\begin{proposition}\label{prop:dell-ext-vanishing}
Let $T\in\Proj\E$, let $Y\in\E$, and let $r\geq1$. Assume
\[
 r\text{-}\dell_{\E}Y\leq a<\infty
\]
and
\[
 \Ext_{\Acat}^{j}(T,P)=0
 \qquad(P\in\Proj\E,\ 2\leq j\leq a+r+1).
\]
Then
\[
 \Ext_{\Acat}^{r+1}(T,Y)=0.
\]
\end{proposition}

\begin{proof}
Put $b=r\text{-}\dell_{\E}Y\leq a$. Choose a stable retraction
\[
 \Omega_{\E}^{b}Y\rightleftarrows\Omega_{\E}^{b+r}Z.
\]
By \cref{lem:ambient-dimension-shift},
\[
 \Ext_{\Acat}^{r+1}(T,Y)
 \simeq
 \Ext_{\Acat}^{b+r+1}(T,\Omega_{\E}^{b}Y).
\]
The functor $\Ext_{\Acat}^{b+r+1}(T,-)$ annihilates morphisms factoring through relative projectives by the assumed vanishing. Hence the stable retraction induces a split monomorphism into
\[
 \Ext_{\Acat}^{b+r+1}(T,\Omega_{\E}^{b+r}Z),
\]
which is zero by \cref{lem:ambient-dimension-shift} with $s=1$ and $q=b+r$.
\end{proof}

\begin{theorem}\label{thm:ambient-ext}
Let $T\in\Proj\E$, let $X\in\E$, and let $k\geq1$. Suppose
\[
 m=k\text{-}\ddell_{\E}X<\infty
\]
and
\[
 \Ext_{\Acat}^{j}(T,P)=0
 \qquad(P\in\Proj\E,\ 2\leq j\leq m+k+1).
\]
Then
\[
 \Ext_{\Acat}^{k+1}(T,X)=0.
\]
\end{theorem}

\begin{proof}
Choose an admissibly exact sequence
\[
 0\to C_t\to\cdots\to C_0\to X\to0,
 \qquad t\leq m,
\]
with $(i+k)\text{-}\dell_{\E}C_i\leq m-i$. Write it as conflations
\[
 0\to Z_{i+1}\to C_i\to Z_i\to0
 \qquad(0\leq i<t),
\]
where $Z_0=X$ and $Z_t=C_t$. Put
\[
 r_i=i+k,
 \qquad
 b_i=r_i\text{-}\dell_{\E}C_i\leq m-i.
\]
Then
\[
 b_i+r_i+1\leq m+k+1,
\]
so \cref{prop:dell-ext-vanishing} gives
\[
 \Ext_{\Acat}^{i+k+1}(T,C_i)=0
 \qquad(0\leq i\leq t).
\]
Applying $\Hom_{\Acat}(T,-)$ to the conflations gives injections
\[
 \Ext_{\Acat}^{i+k+1}(T,Z_i)
 \longrightarrow
 \Ext_{\Acat}^{i+k+2}(T,Z_{i+1}).
\]
Iteration yields
\[
 \Ext_{\Acat}^{k+1}(T,X)
 \hookrightarrow
 \Ext_{\Acat}^{k+t+1}(T,C_t)=0.\qedhere
\]
\end{proof}

For $k=1$, \cref{thm:ambient-ext} implies that $\ddell_{\E}X<\infty$ and $\Ext_{\Acat}^j(T,P)=0$ for all $P\in\Proj\E$, $j\geq2$ force $\Ext_{\Acat}^2(T,X)=0$. Since $\ddell_{\E}X\leq\dell_{\E}X$, this contains the finite-relative-delooping implication of \cite[Proposition~4.4]{ChenLiZhangZhao}.

\section{Derived delooping level with support  $\tau$-tilting modules}\label{sec:relative}

Let $A$ be a finite-dimensional basic algebra over an algebraically closed field.
\begin{definition}\label{def:tau-tilting}
	{\rm(\cite{AIR})}
	A module $T\in\modu A$ is $\tau$-rigid if $\Hom_A(T,\tau T)=0$. It is $\tau$-tilting if it is $\tau$-rigid and the number $|T|$ of isomorphism classes of indecomposable direct summands equals $|A|$. It is support $\tau$-tilting if it is $\tau$-tilting over $A/AeA$ for some idempotent $e$. We call $T$ self-orthogonal if
	\[
	\Ext_A^i(T,T)=0\qquad(i\geq1).
	\]
\end{definition}

\begin{proposition}\label{prop:fac-projectives}
	Let $T$ be a support $\tau$-tilting right $A$-module, put $I=\ann_A T$ and $\barA=A/I$. Then $T$ is a $1$-tilting $\barA$-module and $\Fac_A T=\Fac_{\barA}T$ with the same induced exact structure. The exact category $\Fac T$ has enough projective objects and
	\[
	\Proj(\Fac T)=\add T.
	\]
	Moreover, $D({}_{\barA}\barA)\in\Fac T$.
\end{proposition}

\begin{proof}
	This is the standard support $\tau$-tilting reduction together with the classical tilting theorem; see \cite[Proposition~2.2]{AIR} and \cite[Chapter~VI, Section~2]{ASS}. The equality of the induced exact structures follows because all terms are annihilated by $I$.
\end{proof}

\begin{proposition}\label{prop:tor-criterion}
	Let $T$ be a $\tau$-tilting right $A$-module, let $I=\ann_A T$, and put $\barA=A/I$. Then $I$ is nilpotent, $T$ is a $1$-tilting $\barA$-module, and the following are equivalent:
	\begin{enumerate}
		\item $T$ is a $1$-tilting $A$-module;
		\item $\Ext_A^2(T,\Fac T)=0$;
		\item $\Ext_A^2(T,D({}_{\barA}\barA))=0$;
		\item $\Tor_1^A(T_A,{}_{A}I)=0$.
	\end{enumerate}
\end{proposition}

\begin{proof}
	This is the right-module form of \cite[Theorems~3.5 and~3.6]{ChenLiZhangZhao}, obtained from the left-module formulation by passage to opposite algebras.
\end{proof}

\begin{proposition}\label{prop:canonical-semibrick}
	Let $T=T_1\oplus\cdots\oplus T_n$ be a basic support $\tau$-tilting right $A$-module, put
	\[
	B=\End_A(T),\qquad \Ccat=\Fac T,
	\]
	and let $S=S_1\oplus\cdots\oplus S_m$ be the canonical semibrick with
	\[
	\Ccat=\Filt(\Fac S).
	\]
	Then
	\[
	F=\Hom_A(T,-):\Ccat\longrightarrow\Sub(DT)
	\]
	is an exact equivalence and induces a bijection between $S_1,\ldots,S_m$ and the simple $B$-modules belonging to $\Sub(DT)$. Put $P_j=F(T_j)$ and $L_j=\topm P_j$. After renumbering, for $m<j\leq n$ one has
	\[
	L_j\notin\Sub(DT),\qquad \rad P_j\in\Sub(DT).
	\]
	Consequently, there is a unique $N_j\in\Ccat$, up to isomorphism, such that
	\[
	F(N_j)\simeq\rad P_j.
	\]
	Under $F$, $\add T$ corresponds to $\operatorname{proj}B$.
\end{proposition}

\begin{proof}
	See \cite[Theorem~4.1 and Definition~4.4]{XuZhang}, together with the Brenner--Butler equivalence over $\barA$.
\end{proof}

\begin{definition}\label{def:depth}
	{\rm(\cite[Definition~3.1]{XuZhang})}
	With the notation of \cref{prop:canonical-semibrick}, define
	\[
	\grade_T M=\inf\{i\geq0\mid \Ext_A^i(M,T)\neq0\}
	\qquad(M\in\Ccat)
	\]
	and
	\[
	\depth_T\Ccat=\sup_{1\leq i\leq m}\grade_T S_i.
	\]
\end{definition}

Since $\ann_A T$ acts trivially on $T$, one has $\End_A(T)=\End_{\barA}(T)$. The conventions from Section~2 include the case $T=0$: then $B=0$, $\Ccat=0$, $m=n=0$, and every algebra, depth and canonical relative invariant below is zero.

\begin{definition}\label{def:ordinary-relative}
{\rm(\cite[Definition~4.4]{XuZhang})}
Let $\Ccat=\Fac T$ and $k\geq1$. 
The canonical relative $k$-delooping level of $\Ccat$ with respect to $T$ is
\begin{equation}\label{eq:relative-dell}
 k\text{-}\dell_T\Ccat
 =\max\left\{
 \begin{array}{ll}
 k\text{-}\dell_{\Ccat} S_i,&1\leq i\leq m,\\[2pt]
 1+(k+1)\text{-}\dell_{\Ccat} N_j,&m<j\leq n.
 \end{array}
 \right.
\end{equation}
If $m=n$, the second family is omitted. The global objectwise relative level is
\[
 \operatorname{gl}\text{-}k\text{-}\dell_T\Ccat
 =\sup\{k\text{-}\dell_{\Ccat} M\mid M\in\Ccat\}.
\]
\end{definition}

\begin{definition}\label{def:relative-ddell}
Let $\Ccat=\Fac T$ and $k\geq1$.
The \emph{canonical relative $k$-derived delooping level} of $\Ccat$ with respect to $T$ is
\begin{equation}\label{eq:relative-ddell}
 k\text{-}\ddell_T\Ccat
 =\max\left\{
 \begin{array}{ll}
 k\text{-}\ddell_{\Ccat} S_i,&1\leq i\leq m,\\[2pt]
 1+(k+1)\text{-}\ddell_{\Ccat} N_j,&m<j\leq n.
 \end{array}
 \right.
\end{equation}
For $k=1$, we write  $\ddell_T\Ccat$.
\end{definition}

\smallskip
\noindent\emph{Independence of choices.}
By Krull--Schmidt uniqueness, the indecomposable summands of a basic support $\tau$-tilting module are determined up to permutation and isomorphism. The canonical semibrick is unique in the same sense \cite[Theorem~4.1]{XuZhang}; under the equivalence $F$, the projectives $P_j=F(T_j)$ and their radicals are therefore determined up to isomorphism, and each $N_j$ with $F(N_j)\simeq\rad P_j$ is unique up to isomorphism. Thus the maxima in \eqref{eq:relative-dell} and \eqref{eq:relative-ddell} depend only on the isomorphism class of the basic representative of $T$.

\begin{remark}\label{rem:regular-case}
If $T=A$, then $\Ccat=\modu A$, the canonical semibrick is the direct sum of the simple $A$-modules, and the family $N_j$ is empty. Hence
\[
 k\text{-}\ddell_A(\modu A)=k\text{-}\ddell A,
 \qquad
 k\text{-}\dell_A(\modu A)=k\text{-}\dell A.
\]
Thus \cref{def:relative-ddell} recovers the derived delooping level of Guo and Igusa.
\end{remark}

\begin{proposition}\label{prop:ambient-relative-one-way}
For every $X\in\Scat=\Sub(DT)$ and $k\geq1$,
\begin{equation*}\label{eq:ambient-relative-one-way}
 k\text{-}\dell_BX\leq k\text{-}\dell_{\Scat}X,
 \qquad
 k\text{-}\ddell_BX\leq k\text{-}\ddell_{\Scat}X.
\end{equation*}
\end{proposition}

\begin{proof}
By \cref{prop:canonical-semibrick}, $\Proj\Ccat=\add T$ corresponds to $\operatorname{proj}B$. Hence the projective objects of $\Scat$ are precisely the projective $B$-modules. The inclusion
\[
 \iota:\Scat\longrightarrow\modu B
\]
is exact. Apply \cref{thm:transfer}, the result can be obtained.
\end{proof}

\begin{theorem}\label{thm:relative-main}
With the preceding notation, for every $k\geq1$,
\begin{equation*}\label{eq:relative-main}
 k\text{-}\ddell B\leq k\text{-}\ddell_T\Ccat
 \leq k\text{-}\dell_T\Ccat.
\end{equation*}
For $k=1$,
\begin{equation*}\label{eq:dimension-chain}
 \depth_T\Ccat\leq\findim B^{\op}\leq\Findim B^{\op}
 \leq\ddell B\leq\ddell_T\Ccat\leq\dell_T\Ccat.
\end{equation*}
\end{theorem}

\begin{proof}
Let $L_1,\ldots,L_n$ be the simple right $B$-modules which are the tops of the indecomposable projectives $P_j=F(T_j)$.

For $1\leq i\leq m$, \cref{prop:canonical-semibrick} gives
$L_i\simeq F(S_i)\in\Scat$. By \cref{thm:transfer,prop:ambient-relative-one-way},
\begin{equation*}\label{eq:inside-simple}
 k\text{-}\ddell_B L_i
 \leq k\text{-}\ddell_{\Scat}L_i
 =k\text{-}\ddell_{\Ccat} S_i.
\end{equation*}
For $m<j\leq n$, the projective-cover sequence is
\[
 0\longrightarrow\rad P_j\longrightarrow P_j\longrightarrow L_j\longrightarrow0,
\]
so
\begin{equation*}\label{eq:boundary-syzygy}
 \Omega_BL_j\simeq\rad P_j\simeq F(N_j).
\end{equation*}
By \cref{lem:ddell-syzygy-step},
\[
 k\text{-}\ddell_BL_j
 \leq1+(k+1)\text{-}\ddell_B(\rad P_j).
\]
Using \cref{prop:ambient-relative-one-way,thm:transfer},
\[
 (k+1)\text{-}\ddell_B(\rad P_j)
 \leq(k+1)\text{-}\ddell_{\Scat}F(N_j)
 =(k+1)\text{-}\ddell_{\Ccat}N_j.
\]
Therefore
\begin{equation*}\label{eq:outside-simple}
 k\text{-}\ddell_BL_j
 \leq1+(k+1)\text{-}\ddell_{\Ccat}N_j.
\end{equation*}
Thus the first inequality holds. Xu and Zhang proved
$\depth_T\Ccat\leq\findim B^{\op}$ \cite[Theorem~3.4]{XuZhang}. The inequality
$\findim B^{\op}\leq\Findim B^{\op}$ is immediate, and Guo and Igusa proved
$\Findim B^{\op}\leq\ddell B$ \cite[Theorem~1.1]{GuoIgusa}. Combining these with the first inequality gives the second inequality.
\end{proof}

\begin{remark}\label{rem:relative-main-literature}
For $T=A$,  \cref{thm:relative-main} reduces to the Guo--Igusa inequality $\Findim A^{\op}\leq\ddell A\leq\dell A$ \cite[Theorem~1.1]{GuoIgusa}.
\end{remark}

\begin{remark}\label{rem:test-objects}
The test objects in \cref{def:relative-ddell} encode all simple $B$-modules. The bricks $S_i$ account for the simples contained in $\Sub(DT)$, whereas the modules $N_j$ account for the radicals of projective covers of the remaining simples. Thus $\ddell_T\Ccat$ is not the supremum of objectwise relative derived delooping levels; it is the canonical quantity required to control the opposite endomorphism algebra. The boundary contribution can be genuinely nonprojective; see \cref{ex:nontrivial-boundary}.
\end{remark}

We next identify the role of an injective cogenerator.

\begin{definition}\label{def:finite-sum-cogenerator}
Let $\E$ be an exact category. An object $I\in\E$ is a \emph{finite-sum admissible cogenerator} if, for every $X\in\E$, there are an integer $r\geq1$ and an inflation
\[
 X\hookrightarrow I^{\oplus r}.
\]
Define
\begin{equation*}\label{eq:gddell}
 \gddell(\E)=\sup\{\ddell_{\E}X\mid X\in\E\}.
\end{equation*}
\end{definition}

\begin{theorem}\label{thm:cogenerator}
Let $I$ be a finite-sum admissible cogenerator of $\E$. Then
\begin{equation*}\label{eq:cogenerator-bound}
 \ddell_{\E}I\leq\gddell(\E)
 \leq\ddell_{\E}I+1.
\end{equation*}
Consequently, the following are equivalent:
\begin{enumerate}
\item $\ddell_{\E}I<\infty$;
\item $\gddell(\E)<\infty$;
\item $\Fddell(\E)=\E$.
\end{enumerate}
\end{theorem}

\begin{proof}
The first inequality is immediate. If $\ddell_{\E}I=\infty$, the upper inequality is tautological. Assume therefore that $d=\ddell_{\E}I<\infty$, and let $X\in\E$. Choose an inflation
\[
 X\hookrightarrow I^{\oplus r}.
\]
By \cref{lem:finite-sums}, $\ddell_{\E}(I^{\oplus r})\leq d$, and \cref{cor:subobject} gives
$\ddell_{\E}X\leq d+1$. Taking the supremum, we have  the inequality.
\end{proof}

\begin{proposition}\label{prop:DbarA-cogenerator}
Let $T$ be a support $\tau$-tilting right $A$-module, put
$\barA=A/\ann_A T$ and $\Ccat=\Fac T$. Then
$D({}_{\barA}\barA)$ is an injective finite-sum admissible cogenerator of $\Ccat$.
\end{proposition}

\begin{proof}
By \cref{prop:fac-projectives}, $T$ is a $1$-tilting $\barA$-module. In $\modu\barA$, its tilting torsion class is
\[
 \Fac T=T^{\perp_1}
 =\{X\in\modu\barA\mid\Ext_{\barA}^1(T,X)=0\}.
\]
Since $D({}_{\barA}\barA)$ is injective, it belongs to $\Fac T$. For $X\in\Fac T$, choose a monomorphism
\[
 X\longrightarrow \bigl(D({}_{\barA}\barA)\bigr)^{\oplus r}
\]
in $\modu\barA$. Its cokernel belongs to $\Fac T$ because $\Fac T$ is quotient closed. Hence the monomorphism is an inflation for the exact structure induced on $\Ccat$.
\end{proof}

\begin{theorem}\label{thm:DbarA-control}
Let $T$ be a support $\tau$-tilting right $A$-module, put
$\barA=A/\ann_A T$ and $\Ccat=\Fac T$. Then
\begin{equation*}\label{eq:DbarA-control}
 \ddell_{\Ccat}(D({}_{\barA}\barA))
 \leq\gddell(\Ccat)
 \leq\ddell_{\Ccat}(D({}_{\barA}\barA))+1.
\end{equation*}
In particular,
\[
 \ddell_{\Ccat}(D({}_{\barA}\barA))<\infty
 \Longleftrightarrow
 \gddell(\Ccat)<\infty
 \Longleftrightarrow
 \Fddell(\Ccat)=\Ccat.
\]
\end{theorem}

\begin{proof}
Apply \cref{thm:cogenerator} to \cref{prop:DbarA-cogenerator}.
\end{proof}

\begin{remark}\label{rem:cogenerator-literature}
The injective module $D({}_{\barA}\barA)$ is the same distinguished object that appears in the Ext--Tor characterization of $1$-tilting modules in \cite[Theorems~3.5 and 3.6]{ChenLiZhangZhao}. \Cref{thm:DbarA-control} supplies an exact-categorical explanation: finiteness for this single cogenerator is equivalent, up to an additive constant one, to a uniform derived-delooping bound on the whole torsion class $\Fac T$.
\end{remark}


\begin{assumption}\label{ass:localization}
Let $\Acat$ be an essentially small idempotent-complete exact category with enough projective objects. Let $\Xcat\subseteq\Acat$ be a full resolving subcategory, equipped with the exact structure restricted from $\Acat$. Assume that the inclusion identifies the full subcategory $\Proj\Xcat$ with $\Proj\Acat\subseteq\Xcat$ and that, for some integer $c\geq0$,
\begin{equation*}\label{eq:c-syzygy-in-X}
 \Omega_{\Acat}^{c}M\in\Xcat
 \qquad\text{for every }M\in\Acat.
\end{equation*}
\end{assumption}


\begin{remark}Suppose that $\Acat=\modu B$ and that $\Xcat$ is extension closed, submodule closed and contains $\operatorname{proj}B$. The kernel of a projective cover is a submodule of a projective module. Thus the  \Cref{ass:localization} is automatic for the Brenner--Butler torsion-free class $\Sub(DT)$ with $c=1$.
\end{remark}

\begin{lemma}\label{lem:ordinary-localization}
Under \cref{ass:localization}, let $M\in\Acat$ and $r\geq1$. Then
\begin{equation*}\label{eq:ordinary-localization}
 r\text{-}\dell_{\Xcat}(\Omega_{\Acat}^{c}M)
 \leq r\text{-}\dell_{\Acat}M.
\end{equation*}
\end{lemma}

\begin{proof}
Put $a=r\text{-}\dell_{\Acat}M<\infty$ and choose an ambient stable retraction
\[
 \Omega_{\Acat}^{a}M
 \rightleftarrows
 \Omega_{\Acat}^{a+r}N.
\]
Applying $\Omega_{\Acat}^{c}$ gives
\[
 \Omega_{\Acat}^{a+c}M
 \rightleftarrows
 \Omega_{\Acat}^{a+r+c}N.
\]
By \cref{ass:localization}, both $\Omega_{\Acat}^{c}M$ and $\Omega_{\Acat}^{c}N$ belong to $\Xcat$. Since $\Xcat$ is resolving and has the same projective objects as $\Acat$, the projective resolutions of objects of $\Xcat$ compute their relative syzygies. The displayed retraction is therefore precisely a relative stable retraction
\[
 \Omega_{\Xcat}^{a}(\Omega_{\Acat}^{c}M)
 \rightleftarrows
 \Omega_{\Xcat}^{a+r}(\Omega_{\Acat}^{c}N).
\]\end{proof}

\begin{lemma}\label{lem:derived-localization-syzygy}
Under \cref{ass:localization}, for every $M\in\Acat$ and $k\geq1$,
\begin{equation*}\label{eq:derived-localization-syzygy}
 k\text{-}\ddell_{\Xcat}(\Omega_{\Acat}^{c}M)
 \leq k\text{-}\ddell_{\Acat}M.
\end{equation*}
\end{lemma}

\begin{proof}
Let $m=k\text{-}\ddell_{\Acat}M<\infty$ and choose an admissibly exact sequence
\[
 0\to C_t\to\cdots\to C_0\to M\to0,
 \qquad t\leq m,
\]
with $(i+k)\text{-}\dell_{\Acat}C_i\leq m-i.$
By \cref{lem:lift-witness}, applied in $\Acat$ with $s=c$, there is an admissibly exact sequence
\begin{equation}\label{eq:localized-witness}
 0\to C_t^{[c]}\to\cdots\to C_0^{[c]}
 \to\Omega_{\Acat}^{c}M\to0
\end{equation}
with $C_i^{[c]}\simeq\Omega_{\Acat}^{c}C_i$ stably.
By  \cref{ass:localization}, $\Omega_{\Acat}^{c}C_i\in\Xcat$. Schanuel comparison gives projectives $P_i,Q_i$ with
$ C_i^{[c]}\oplus P_i
 \simeq
 \Omega_{\Acat}^{c}C_i\oplus Q_i.$
The right-hand side belongs to $\Xcat$. Since a resolving subcategory contains the  projectives and is closed under direct summands, $C_i^{[c]}\in\Xcat$.
Write \eqref{eq:localized-witness} in $\Acat$ as conflations
\[
 0\to Z_{i+1}^{[c]}\to C_i^{[c]}\to Z_i^{[c]}\to0,
\]
with $Z_0^{[c]}=\Omega_{\Acat}^{c}M$. The latter belongs to $\Xcat$. Since $C_0^{[c]}\in\Xcat$ and $\Xcat$ is resolving, the kernel $Z_1^{[c]}$ belongs to $\Xcat$. Induction gives $Z_i^{[c]}\in\Xcat$ for every $i$, so every conflation in \eqref{eq:localized-witness} lies in the restricted exact structure of $\Xcat$.
Finally, \cref{lem:ordinary-localization} and stable invariance give
\[
 (i+k)\text{-}\dell_{\Xcat}C_i^{[c]}
 \leq(i+k)\text{-}\dell_{\Acat}C_i
 \leq m-i.
\]
Thus \eqref{eq:localized-witness} is a relative $k$-derived delooping level of value $m$.
\end{proof}

\begin{theorem}\label{thm:localization}
Under \cref{ass:localization}, for every $X\in\Xcat$ and every $k\geq1$,
\begin{equation*}\label{eq:localization-main}
 k\text{-}\ddell_{\Acat}X
 \leq k\text{-}\ddell_{\Xcat}X
 \leq c+(k+c)\text{-}\ddell_{\Acat}X.
\end{equation*}
\end{theorem}

\begin{proof}
Let $\iota:\Xcat\hookrightarrow\Acat$ denote the inclusion. By \cref{ass:localization}, $\iota$ is exact and sends relative projective objects to  projective objects. Hence $\pd(\iota)=0$, and \cref{thm:transfer} gives
\begin{equation*}\label{eq:localization-left}
 k\text{-}\ddell_{\Acat}X\leq k\text{-}\ddell_{\Xcat}X.
\end{equation*}
By \cref{lem:ddell-syzygy-step}, we have
\[
(k+r)\text{-}\ddell_{\Xcat}(\Omega_{\Xcat}^{r}X)
\leq1+(k+r+1)\text{-}\ddell_{\Xcat}(\Omega_{\Xcat}^{r+1}X).
\]
By induction , we have
\begin{equation*}\label{eq:localization-intermediate}
 k\text{-}\ddell_{\Xcat}X
 \leq c+(k+c)\text{-}\ddell_{\Xcat}(\Omega_{\Xcat}^{c}X).
\end{equation*} Since the projective objects of the two exact categories coincide, \cref{lem:schanuel} gives a stable isomorphism
\begin{equation*}\label{eq:relative-ambient-c-syzygy}
 \Omega_{\Xcat}^{c}X\simeq\Omega_{\Acat}^{c}X
 \qquad\text{in }\stable{\Xcat}.
\end{equation*}
The object on the right lies in $\Xcat$. By \cref{prop:basic-properties}(2),  and \cref{lem:derived-localization-syzygy} with, we obtain
\[
 (k+c)\text{-}\ddell_{\Xcat}(\Omega_{\Xcat}^{c}X)
 =(k+c)\text{-}\ddell_{\Xcat}(\Omega_{\Acat}^{c}X)
 \leq(k+c)\text{-}\ddell_{\Acat}X.
\]
Combining the above inequality, the proof is complete.
\end{proof}

We specialize to the Brenner--Butler torsion-free class.

\begin{proposition}\label{prop:SubDT-resolving}
Let $T$ be a support $\tau$-tilting right $A$-module and $B=\End_A(T)$. Then
\[
 \Scat=\Sub(DT)\subseteq\modu B
\]
is a resolving subcategory, its projective objects are precisely $\operatorname{proj}B$, and
\begin{equation*}\label{eq:first-syzygies-SubDT}
 \Omega_B(\modu B)\subseteq\Scat.
\end{equation*}
\end{proposition}

\begin{proof}
The category $\Scat$ is extension closed and submodule closed. Under the exact equivalence of \cref{prop:canonical-semibrick}, its projective objects correspond to $\add T$, and hence are exactly the projective $B$-modules. Thus $\Scat$ is resolving. If
\[
 0\to\Omega_BM\to P_M\to M\to0
\]
is a projective-cover sequence, then $P_M\in\Scat$ and the submodule $\Omega_BM$ belongs to $\Scat$.
\end{proof}

\begin{corollary}\label{cor:SubDT-localization}
Let $T$ be a support $\tau$-tilting right $A$-module and $B=\End_A(T)$. For $X\in\Sub(DT)$ and every $k\geq1$, we have
\begin{equation*}\label{eq:SubDT-localization}
 k\text{-}\ddell_BX
 \leq k\text{-}\ddell_{\Sub(DT)}X
 \leq1+(k+1)\text{-}\ddell_BX.
\end{equation*}
\end{corollary}

\begin{proof}
Apply \cref{thm:localization} with $c=1$ and \cref{prop:SubDT-resolving}.
\end{proof}

We finally connect the relative theory to $\tau$-tilting modules. The following result records the finite range of self-orthogonality actually used by the argument.

\begin{theorem}\label{thm:finite-range-tor}
Let $T$ be a $\tau$-tilting right $A$-module, put
\[
 I=\ann_AT,\qquad \barA=A/I,\qquad \Ccat=\Fac T,
\]
and let $k\geq1$. Assume
\[
 m=k\text{-}\ddell_{\Ccat}(D({}_{\barA}\barA))<\infty
\]
and
\begin{equation*}\label{eq:finite-range-orthogonality}
 \Ext_A^j(T,T)=0
 \qquad(2\leq j\leq m+k+1).
\end{equation*}
Then
\begin{equation*}\label{eq:Tor-vanishing}
 \Tor_k^A(T_A,{}_{A}I)=0.
\end{equation*}
If $k=1$, then $T$ is a $1$-tilting $A$-module.
\end{theorem}

\begin{proof}
By \cref{prop:fac-projectives}, the induced exact category $\Ccat$ has
$\Proj\Ccat=\add T$ and contains $D({}_{\barA}\barA)$. By assumption,
\[
 \Ext_A^j(T,P)=0
 \qquad(P\in\Proj\Ccat,\ 2\leq j\leq m+k+1).
\]
Apply \cref{thm:ambient-ext}, we obtain
 $\Ext_A^{k+1}\bigl(T_A,D({}_{A}\barA)\bigr)=0.$
For right $A$-modules, the standard duality isomorphism is
$\Ext_A^{k+1}\bigl(T_A,D({}_{A}\barA)\bigr)
 \simeq D\Tor_{k+1}^A(T_A,{}_{A}\barA).$
Applying $T_A\otimes_A-$ to the exact sequence of left $A$-modules
\[
 0\longrightarrow{}_{A}I\longrightarrow{}_{A}A
 \longrightarrow{}_{A}\barA\longrightarrow0
\]
and using the projectivity of ${}_{A}A$ gives
\begin{equation*}\label{eq:Tor-shift-ann}
 \Tor_{k+1}^A(T_A,{}_{A}\barA)
 \simeq\Tor_k^A(T_A,{}_{A}I)=0
 \qquad(k\geq1).
\end{equation*}
In particular,  \cref{prop:tor-criterion} implies that $T$ is $1$-tilting for $k=1$.
\end{proof}

\begin{corollary}\label{cor:selforth-tilting}
Let $T$ be a self-orthogonal $\tau$-tilting right $A$-module, put $I=\ann_A T$, $\barA=A/I$ and $\Ccat=\Fac T$.
If $k\text{-}\ddell_{\Ccat}(D({}_{\barA}\barA))<\infty$, then $\Tor_k^A(T,I)=0$. In particular,
\[
 \ddell_{\Ccat}(D({}_{\barA}\barA))<\infty
 \quad\Longrightarrow\quad T\text{ is }1\text{-tilting}.
\]
The same conclusion follows from $\gddell(\Fac T)<\infty$.
\end{corollary}

\begin{theorem}\label{thm:ambient-tilting}
Let $T$ be a $\tau$-tilting right $A$-module and $B=\End_A(T)$. 
If $d=2\text{-}\ddell_B(DT)<\infty$ and
$\Ext_A^j(T,T)=0$ for $2\leq j\leq d+3,
$
then $T$ is a $1$-tilting $A$-module. In particular, the conclusion holds whenever $T$ is self-orthogonal.

More generally, if $q\geq2$, $d_q=q\text{-}\ddell_B(DT)<\infty$, and
$\Ext_A^j(T,T)=0$ for $2\leq j\leq d_q+3$, then $T$ is $1$-tilting.
\end{theorem}

\begin{proof}
By \cref{cor:SubDT-localization} with $k=1$,
\[
 \ddell_{\Sub(DT)}(DT)
 \leq1+2\text{-}\ddell_B(DT)=d+1.
\]
The exact equivalence \cref{prop:canonical-semibrick} sends $D({}_{\barA}\barA)$ to $DT$, indeed, $$F(D(_{\barA}\barA))=\Hom_A(T,D\barA)\cong \Hom_{\barA}(T,D\barA)\cong DT,~~f\mapsto[t\mapsto f(t)(1)].$$
So \cref{thm:transfer} gives
\[
 m:=\ddell_{\Fac T}(D({}_{\barA}\barA))\leq d+1.
\]
which implies
$\Ext_A^j(T,T)=0$ for $2\leq j\leq m+2$. Apply \cref{thm:finite-range-tor}, we have  $T$ is a $1$-tilting $A$-module.

If $q\geq2$, repeated use of \cref{lem:monotonicity} gives
$2\text{-}\ddell_B(DT)\leq d_q$, and the same argument applies.
\end{proof}


\section{Examples}\label{sec:examples}

\begin{example}\label{ex:nontrivial-boundary}
Let $A$ be the cyclic Nakayama algebra
\[
 \begin{tikzcd}[column sep=4em]
 1\arrow[r,bend left=18,"\alpha"]&2\arrow[l,bend left=18,"\beta"]
 \end{tikzcd}
 \qquad\text{with}\qquad
 \alpha\beta\alpha=\beta\alpha\beta=0.
\] 
The Auslander--Reiten quiver may be displayed as
\[
\begin{tikzcd}[column sep=1.55em,row sep=0.4em]
	& P_1\arrow[dr] && P_2\arrow[dr] &\\
	M_2\arrow[dr]\arrow[ur]&&M_1\arrow[dr]\arrow[ur]&&M_2\\
		& S_2\arrow[ur] &&S_1\arrow[ur] &\\
\end{tikzcd}
\] where $P_i=e_iA$, $M_i=P_i/S_i$ and  the two occurrences of $M_2$ are identified. 

Let $T=P_1\oplus S_1$ be a basic $\tau$-tilting $A$-module. Then $$
 \Ccat=\Fac T=\add\{P_1,S_1,M_1\}=\Filt(\Fac M_1).$$
The quotient and socle maps give a conflation
\begin{equation*}\label{eq:boundary-example-conflation}
 0\longrightarrow S_1\xrightarrow{}P_1
 \xrightarrow{}M_1\longrightarrow0.
\end{equation*}
Consequently, $\Omega_{ \Ccat}M_1\cong S_1$. Every object of $\Ccat$ has relative projective dimension at most one, whereas $M_1$ is not relative projective. Therefore, for every $k\geq1$,
$k\text{-}\dell_{\Ccat}M_1=k\text{-}\ddell_{\Ccat}M_1=1.$

The endomorphism algebra of $T$ is $B=\End_{A}(T)\simeq KQ'/\langle ba\rangle,$
where  \[
Q':\quad
\begin{tikzcd}[column sep=4em]
	1\arrow[r,bend left=18,"a"]&2\arrow[l,bend left=18,"b"]
\end{tikzcd},
\]
Let $Q_i=e_iB$ and $L_i=\topm Q_i$ ($i=1,2$). Then
\[
 \rad Q_1\simeq Q_2,
 \qquad
 \rad Q_2\simeq L_1,
\]
and the simple modules have minimal projective resolutions
\[
 0\longrightarrow Q_2\longrightarrow Q_1\longrightarrow L_1\longrightarrow0,
\]
\[
 0\longrightarrow Q_2\longrightarrow Q_1\longrightarrow Q_2
 \longrightarrow L_2\longrightarrow0.
\]
Hence $\pd_{B}L_1=1$, $\pd_{B}L_2=2$ and $\gldim B=2$. Renaming the opposite arrows yields the same presentation, so $B^{\op}\simeq B$. Let $F=\Hom_{A}(T,-)$. Since $\dim_K\Hom_{A}(P_1,M_1)=1,$
 and $\Hom_{A}(S_1,M_1)=0,$
the $B$-module $F(M_1)$ has dimension vector $(1,0)$ and therefore $
 F(M_1)\simeq L_1\simeq\rad Q_2=\rad F(T_2)$ which imples 
 $N_2=M_1$. It follows that
\[
 \ddell_{T}\Ccat
 =\max\{\ddell_{\Ccat}M_1,1+2\text{-}\ddell_{\Ccat}M_1\}=2,
 \qquad
 \dell_{T}\Ccat=2.
\]
Moreover, the relative projectives $P_1,S_1$ have derived level zero and $M_1$ has level one, so
\[
 \gddell(\Ccat)=1<2=\ddell_{T}\Ccat.
\]
Finally,
\[
 \findim B^{\op}=\Findim B^{\op}
 =\ddell B=\ddell_{T}\Ccat=\dell_{T}\Ccat=2.
\]
\end{example}

For the final example we use the elementary product observation that projective objects, conflations, syzygies and stable morphisms in a finite product of exact categories are computed componentwise. It follows directly, by aligning the positions of stable retractions and padding derived witnesses by split conflations, that ordinary and derived delooping levels are componentwise maxima; the same is true for algebra invariants and for the canonical relative invariants of product $\tau$-tilting pairs.

\begin{example}\label{ex:genuine-strict}
Let $\Lambda_{\mathrm{KR}}$ be the Kershaw--Rickard algebra \cite{KershawRickard}. Guo and Igusa computed
\[
 \ddell\Lambda_{\mathrm{KR}}=1,
 \qquad
 \dell\Lambda_{\mathrm{KR}}=\infty
\]
in \cite{GuoIgusa}. Let $(A,T)$ be the pair from \cref{ex:nontrivial-boundary}, and put
\[
 A'=\Lambda_{\mathrm{KR}}\times A,
 \qquad
 T'=\Lambda_{\mathrm{KR}}\oplus T.
\]
Then $T'$ is a basic $\tau$-tilting $A'$-module,
\[
 \Fac T'=\modu\Lambda_{\mathrm{KR}}\times\Fac T\neq\add T',
 \qquad
 \End_A(T')\simeq\Lambda_{\mathrm{KR}}\times B,
\]
and its canonical test family contains the nonprojective boundary object $M_1$ from \cref{ex:nontrivial-boundary}. By the product observation,
\[
 \ddell_{T'}\Fac T'
 =\max\{\ddell\Lambda_{\mathrm{KR}},\ddell_{T}\Fac T\}=2,
\]
whereas
\[
 \dell_{T'}\Fac T'
 =\max\{\dell\Lambda_{\mathrm{KR}},\dell_{T}\Fac T\}=\infty.
\]
Thus
\[
 \ddell_{T'}\Fac T'<\dell_{T'}\Fac T'.
\]

\end{example}

\section*{Acknowledgments}

Hanpeng Gao is supported by  National Natural Science Foundation of China (No.12301041).
Dajun  Liu  is supported by the National Natural Science Foundation of China (No.12101003), the Key Research Project of Anhui Provincial Department of Education ( No. 2025AHGXZK31616), Anhui Province 2025 New Era Education Quality Engineering Project : Matrix Theory (Grant No. 2025kcszsfkc023). Houjun Zhang is supported by the National Natural Science Foundation of China (No. 12301051).

\section*{Data availability}
Data sharing not applicable to this article as no datasets were generated or analyzed during the current study.

\section*{Declarations}

{\bf Conflict of Interests} The authors has no conflicts of interest to declare that are relevant to the content of this article.

\end{document}